\documentclass[11pt]{amsart}

\usepackage{amsmath,amssymb,amsthm}
\usepackage{upgreek}
\usepackage{amsaddr}
\usepackage{mathtools}
\usepackage{geometry}
\usepackage{graphicx}
\usepackage{booktabs}
\usepackage{microtype}
\usepackage{hyperref}
\usepackage{mathrsfs}

\hypersetup{
    hidelinks,
    pdftitle={Mixed Volumes, the Wigner Caustic, and Isoperimetric Inequalities},
    pdfauthor={Michal Zwierzynski},
    pdfsubject={Convex geometry, mixed volumes, and Wigner caustics},
    pdfkeywords={central symmetral, Wigner caustic, mixed volume,
    isoperimetric inequality, constant width}
}

\allowdisplaybreaks
\numberwithin{equation}{section}

\newtheorem{theorem}{Theorem}[section]
\newtheorem{lemma}[theorem]{Lemma}
\newtheorem{proposition}[theorem]{Proposition}
\newtheorem{corollary}[theorem]{Corollary}

\theoremstyle{definition}
\newtheorem{definition}[theorem]{Definition}

\theoremstyle{remark}
\newtheorem{remark}[theorem]{Remark}

\newcommand{\RR}{\mathbb{R}}
\newcommand{\ZZ}{\mathbb{Z}}
\newcommand{\SSph}{\mathbb{S}}
\newcommand{\Gr}{\operatorname{Gr}}
\newcommand{\Vol}{\operatorname{Vol}}
\newcommand{\Area}{\operatorname{Area}}

\newcommand{\tr}{\operatorname{tr}}
\newcommand{\Id}{\operatorname{Id}}
\newcommand{\Wig}{\mathcal{W}}
\newcommand{\dd}{\,\mathrm{d}}
\newcommand{\restr}{\mathbin{|}}

\newcommand{\heven}{h_{\mathrm{even}}}
\newcommand{\hodd}{h_{\mathrm{odd}}}

\newcommand{\doi}[1]{\href{https://doi.org/#1}{\nolinkurl{doi:#1}}}

\title[The Wigner Caustic and Isoperimetric Inequalities]
{Mixed Volumes, the Wigner Caustic, and Isoperimetric Inequalities}

\author{Micha\l{} Zwierzy\'nski}

\email{Michal.Zwierzynski@pw.edu.pl}
\email{ORCID: 0000-0002-9627-1563}

\address{Warsaw University of Technology\\
Faculty of Mathematics and Information Science\\
ul. Koszykowa 75\\
00-662 Warsaw, Poland}



\begin{document}

\begin{abstract}
Let $K\subset\RR^n$ be a~convex body and
$C=\frac12(K+(-K))$ its central symmetral. We study the~defects
$\mathcal A_k(K)=W_{n-k}(C)-W_{n-k}(K)$ through the~even--odd
decomposition of the~support function. We derive exact even
mixed-volume expansions, identify their Krawtchouk transform with the~
mixed difference-body coefficients, and obtain a~Kubota formula
expressing each defect as the~averaged symmetrization gain of
projections. These identities yield strengthened isoperimetric
inequalities. The~quadratic defect is the~average of the~absolute
oriented areas of projected Wigner caustics and obeys a~sharp
spherical-harmonic stability estimate. We give explicit formulas in
dimensions $2$, $3$, and $4$. The~four-dimensional expansion includes
a~sign-indefinite quartic Wigner volume; on a~plane of cubic harmonics
we find nonspherical constant-width bodies for which it vanishes.
Finally, we prove a~projection-transfer bound and an~
arbitrary-dimensional expansion for constant-width bodies.
\end{abstract}

\maketitle

\section{Introduction}

\noindent
The~planar improved isoperimetric inequality
\begin{equation}
L(K)^2\geqslant
4\uppi\left(\Area(K)+2\left|A^*_{\Wig}(K)\right|\right)
\label{eq:planar-improved}
\end{equation}
relates the~perimeter and area of a~convex body to the~oriented area
of the~Wigner caustic of its boundary.  Equality holds precisely for
bodies of constant width.  This inequality was proved in
\cite[Theorem~3.4]{Zwier2016} and is the~starting point of this paper.
Exact planar relations involving the~same oriented area, the~Constant
Width Measure Set, and the~Spherical Measure Set appear in
\cite[Theorem~2.17]{ZwierRosettes} and
\cite[Theorem~5.1]{ZwierCWMS}. A~normed-plane analogue of the~Wigner caustic and the~Constant Width Measure Set, together
with corresponding improved isoperimetric inequalities for admissible
curves, was developed in
\cite[Sections~4--5]{dosSantosCraizer2022}.   The~singularities and global geometry
of planar Wigner caustics and affine equidistants are treated in
\cite[Section~2]{DomitrzZwierSingular} and
\cite[Sections~3--6]{DomitrzZwier2022}.  The~related Centre Symmetry
Set was introduced by Janeczko as the~bifurcation set of a~family of
ratios \cite{Janeczko1996} and reformulated by Giblin and Holtom as the~envelope of chords joining parallel-tangent pairs
\cite{GiblinHoltom}; its planar geometry is developed further in
\cite{MillerZwier2026}.  Schneider's middle hedgehog
extends the~planar construction beyond smooth ovals
\cite[Sections~2--3]{SchneiderMiddle}, while recent support-function
visualizations are given in \cite{DanielewskaEtAl}.  Convex polygons
with parallel opposite sides, together with their discrete
equidistants and centre symmetry sets, were introduced in
\cite{CraizerTeixeiraSilva}; a~discrete counterpart of
\eqref{eq:planar-improved} for this class was subsequently obtained in
\cite[Theorem~4.4]{KonicerEtAl2026}.  Related Fourier decompositions
and isoperimetric refinements occur in the~work on Hurwitz-type
deficits \cite{CufiGallegoReventos}, higher-order discrete and smooth
inequalities \cite{Kwong2024}, mixed isoperimetric deficits
\cite{Zhang2021}, and higher-order preserving harmonic and midpoint sets
\cite[Theorem~5.2]{SafarewiczZwier2026}.

The~Wigner caustic itself goes back to Berry's phase-space
construction \cite[Section~6]{Berry1977}.  The~role of odd functions
in the~singularity theory of higher-dimensional symplectic Wigner
caustics was studied in \cite{DomitrzManoelRios}.  Affine
equidistants and global centre symmetry sets of Lagrangian
submanifolds are developed further in
\cite{DomitrzRios,DomitrzRiosRuas}; algebraic symmetry-defect loci are
studied in \cite{DiasFarnikJelonek,JaneczkoJelonekRuas}.  The~
convex-hypersurface setting considered here is different, since the~antipodal Gauss directions select one distinguished pair of parallel
tangent hyperplanes.

The~purpose of this paper is to identify the~higher-dimensional
mixed-volume mechanism behind \eqref{eq:planar-improved}.  It is
elementary at its core:
one applies classical inequalities to the~central symmetral
\begin{equation}
C=\frac12(K+(-K)),
\label{eq:central-symmetral-intro}
\end{equation}
whose support function is the~even part of the~support function of
$K$.  The~odd part is not the~support function of a~convex body in
general, but it is a~virtual convex body and hence determines a~hedgehog.
For smooth strictly convex bodies its geometric image is exactly the~Wigner caustic.  The~formal-difference viewpoint was initiated in
\cite{LLR88}; projective hedgehogs and their offset constructions are
also treated in \cite{Rochera2022}; algebraic mixed volumes of
hedgehogs are developed systematically in
\cite[Chapter~3, pp.~67--83]{MMBook}.  In the~planar smooth case the~identity $x_h(u)-x_h(-u)=2x_{\heven}(u)$ shows that the~secant caustic
of $\partial K$ is $2\partial C$.  Thus the~central symmetral is, up to
this homothety and filling the~bounded component, a~special case of
the~secant-caustic construction; we use the~normalization of
\cite[Definition~2.4 and Corollary~4.3]{DomitrzRomeroZwier2021}.
This decomposition turns the~planar correction in
\eqref{eq:planar-improved} into a~general mixed-volume mechanism.

Our main points are as follows.

\begin{enumerate}
\renewcommand{\labelenumi}{\textup{(\roman{enumi})}}

\item
For every $k=1,\ldots,n$, the~difference between $W_{n-k}(K)$ and the~
same quermassintegral of $C$ has an~exact expansion containing only
even numbers of copies of the~Wigner support function.  The~same
defect is the~mean symmetrization gain of the~$k$-dimensional
projections of $K$, which yields positivity and a~sharp
Rogers--Shephard bound.  Moreover, the~mixed volumes of $K$ and $-K$
are, at top degree, exactly the~Krawtchouk transform of these even
Wigner terms.  The~relevant results are
Theorems~\ref{thm:exact-expansion} and~\ref{thm:positive-defects},
and Propositions~\ref{prop:projection-defect}, \ref{prop:RS-bound},
and~\ref{prop:Krawtchouk-transform}.

\item
Applying the~Alexandrov--Fenchel inequalities to $C$ gives a~family
which genuinely strengthens the~generalized Urysohn inequalities for
$K$.  Its $k=2$ member has the~explicit Wigner correction
\begin{equation*}
\frac{2\kappa_n}{\uppi}
\int_{\Gr(n,2)}
\left|A^*_{\Wig}(K\restr E)\right|\dd E.
\end{equation*}
The~family is stated in Theorem~\ref{thm:isoperimetric-family};
its Alexandrov--Fenchel input is clarified in
Proposition~\ref{prop:corrected-AF}, while the~explicit Wigner
interpretation of the~$k=2$ correction follows from
Theorem~\ref{thm:Wigner-projection}.

\item
The~quadratic Wigner correction is a~positive spectral energy of the~
odd part of the~support function.  It controls, with a~sharp constant,
the~squared $L^2$-distance from central symmetry after translations
are removed.  These statements are proved in
Theorem~\ref{thm:spectral-A2} and
Corollary~\ref{cor:stability}.

\item
The~cases $n=2,3,4$ are made explicit.  In dimension four, a~pure
quartic mixed volume $V(\hodd^{[4]})$ appears.  It is the~Gauss-cover
algebraic volume of the~projective Wigner hedgehog and assumes both
signs arbitrarily close to a~ball.  On an~explicit plane of cubic
harmonics we determine the~quartic form, its zero cone, and the~optimal
quartic-to-quadratic ratios.  A~termwise difference-body estimate
sharpens the~lower boundary of the~quadratic--quartic feasible region,
while projection to three dimensions transfers Nishioka's bound
\cite[Theorem~1]{Nishioka2026} to
$\mathcal A_2\leqslant\kappa_4r^2/11$ at width $2r$.
The~dimension-specific formulas are developed in
Sections~\ref{sec:dimension-two}, \ref{sec:dimension-three},
and~\ref{sec:dimension-four}.  The~four-dimensional sign,
cubic-harmonic, and feasible-region results are given in
Lemma~\ref{lem:quartic-signs},
Proposition~\ref{prop:cubic-harmonic-plane}, and
Theorem~\ref{thm:quartic-feasible}, respectively; the~projection
argument uses Theorem~\ref{thm:projection-transfer}.

\item
For constant-width bodies in arbitrary dimension, every
quermassintegral has an~even Wigner expansion.  In even dimensions the~
top-degree projective Wigner volume survives as the~residual term in
the~volume--surface relation.  These results are stated in
Theorem~\ref{thm:CW-expansion} and
Corollary~\ref{cor:even-dimensional-residual}.

\end{enumerate}

There is substantial recent activity around constant-width bodies and
difference-body methods.  Mixed-volume reformulations of the~Blaschke--Lebesgue problem were obtained by Bogosel
\cite{BogoselMixed}; variational methods and Cheeger-type extremal
questions were studied in \cite{BogoselVariational,BogoselCheeger}.
In dimension three, the~support-function formula of Anciaux and
Guilfoyle \cite{AnciauxGuilfoyle} underlies recent spectral progress,
including Nishioka's improved lower bound for the~
Blaschke--Lebesgue problem \cite{Nishioka2026}.  Higher-order
difference bodies and their affine inequalities are developed in
\cite{HaddadEtAl}.  A~termwise generalization of the~higher-order
difference-body inequality, together with its anti-blocking case and
valuation interpretation, appears in \cite{Kotrbaty2025}; a~recent preprint proves
the~original Godbersen mixed-volume inequalities in full
\cite[Theorem~1.1]{KotrbatyMouamine2026}.  Refined mean-width
inequalities for
origin-symmetric bodies appear in \cite{BoroFodorHug}.  These works
address different extremal or normalized problems.  The~contribution
here is the~unified Wigner interpretation of central-symmetral
corrections and the~resulting all-dimensional family.  General
background on constant-width bodies may be found in \cite{MMO}.

Algebraic volumes and
mixed volumes of hedgehogs are established notions; the~novelty claim
concerns the~identities and consequences above, not the~classical
Brunn--Minkowski, Alexandrov--Fenchel, Kubota, or Rogers--Shephard
ingredients; see \cite[Chapter~3, pp.~67--83]{MMBook} for the~algebraic mixed-volume framework for hedgehogs.

Throughout, mixed-volume conventions and the~differential formulas
are those of \cite[Sections~5.1 and~5.5]{Schneider}; the~precise
inequality and integral-geometric references are given where they are
used.

\section{Support functions, hedgehogs, and mixed volumes}

\subsection{Convex bodies and quermassintegrals}

\noindent
Let $K\subset\RR^n$ be a~convex body, i.e. a~compact convex set with
non-empty interior.  Its support function is
\begin{equation*}
h_K(u)=\max_{x\in K}\langle x,u\rangle,
\qquad u\in\SSph^{n-1}.
\end{equation*}
We write $B^n$ for the~Euclidean unit ball and
$\kappa_n=\Vol_n(B^n)$.
The~mixed volume of $n$ convex bodies is the~symmetric coefficient
defined by the~volume polynomial.  Equivalently,
\begin{equation}
V(K_1,\ldots,K_n)
=
\left.
\frac{1}{n!}
\frac{\partial^n}{\partial t_1\cdots\partial t_n}
\Vol_n(t_1K_1+\cdots+t_nK_n)
\right|_{t_1=\cdots=t_n=0}.
\label{eq:mixed-volume-definition}
\end{equation}
See \cite[Section~5.1]{Schneider}.  Repeated arguments are written as
$V(K^{[r]},L^{[s]},B^{[t]})$.
The~quermassintegrals are normalized by
\begin{equation}
W_j(K)=V(K^{[n-j]},B^{[j]}),
\qquad j=0,\ldots,n.
\label{eq:quermass-normalization}
\end{equation}
Thus
\begin{equation}
\Vol_n(K+tB^n)
=
\sum_{j=0}^n\binom{n}{j}W_j(K)t^j,
\label{eq:Steiner}
\end{equation}
and
\begin{equation}
W_0(K)=\Vol_n(K),
\qquad
W_n(K)=\kappa_n,
\qquad
nW_1(K)=S(K),
\label{eq:quermass-special}
\end{equation}
where $\mathcal H^{n-1}$ denotes $(n-1)$-dimensional Hausdorff
measure and $S(K)=\mathcal H^{n-1}(\partial K)$.

We first work with a~convex body of class $C^2_+$, meaning that
$\partial K$ is $C^2$ and has positive Gauss curvature.  Every convex body is a~Hausdorff limit of convex bodies of class
$C^\infty_+$: one applies the~smoothing procedure from
\cite[Theorem~3.4.1]{Schneider} and then adds a~Euclidean ball of
vanishing radius.  Since mixed volumes, and hence quermassintegrals, are
continuous under Hausdorff convergence
\cite[pp.~280--281]{Schneider}, the~inequalities involving only convex
bodies and quermassintegrals extend to arbitrary convex bodies.

For $f\in C^2(\SSph^{n-1})$, put
\begin{equation}
Q_f=\nabla^2_{\SSph^{n-1}}f+f\Id.
\label{eq:Q-def}
\end{equation}
The~associated hedgehog is parametrized by
\begin{equation}
x_f(u)=f(u)u+\nabla_{\SSph^{n-1}}f(u).
\label{eq:hedgehog-param}
\end{equation}
If $f=h_K$, then $Q_f$ is positive definite and
\eqref{eq:hedgehog-param} is the~inverse Gauss-map parametrization of
$\partial K$; see \cite[Section~2.5]{Schneider}.
Conversely, for every smooth $f$, the~operator $Q_{f+R}=Q_f+R\Id$ is
positive definite when $R$ is sufficiently large.  Thus $f+R$ is a~support function and $f=(f+R)-R$ is a~formal difference of support
functions.  This gives a~direct multilinear extension of mixed volume
to all smooth functions.

Let $D_{n-1}$ be the~polarized determinant on symmetric endomorphisms
of an~$(n-1)$-dimensional vector space, normalized by
$D_{n-1}(A,\ldots,A)=\det A$.
The~mixed-volume form extends from support functions of convex bodies
to arbitrary smooth functions by
\begin{equation}
V(f_1,\ldots,f_n)
=
\frac1n
\int_{\SSph^{n-1}}
f_1
D_{n-1}(Q_{f_2},\ldots,Q_{f_n})
\dd\sigma.
\label{eq:mixed-volume-integral}
\end{equation}
Integration by parts makes \eqref{eq:mixed-volume-integral} symmetric
in all arguments.  This is
the~usual mixed volume on convex support functions and its multilinear
extension to virtual convex bodies or hedgehogs; see
\cite[Sections~5.1 and~5.5]{Schneider}.  In particular, in expressions such as
$V(f^{[r]},g^{[s]},1^{[t]})$, the~function $1$ is the~support function
of $B^n$.

If $\ell(u)=\langle a,u\rangle$ is linear, then
\begin{equation}
Q_\ell=0.
\label{eq:linear-kernel}
\end{equation}
Consequently, every mixed volume with one linear argument vanishes.
This is the~support-function form of translation invariance.

\subsection{The~even--odd decomposition}

\noindent
Write $h=h_K$ and decompose
\begin{align}
h=\hodd+\heven,
\quad\text{where}\quad
\heven(u)=\frac{h(u)+h(-u)}2,
\qquad
\hodd(u)=\frac{h(u)-h(-u)}2.
\label{eq:even-odd-decomposition}
\end{align}
Thus $\heven$ is even and $\hodd$ is odd.  The~notation records the~parity explicitly and will be used throughout.

\begin{definition}[Wigner caustic]
\label{def:Wigner-caustic}
For a~convex body $K\subset\RR^n$ of class $C^2_+$, its Wigner
caustic is the~midpoint
locus
\begin{equation}
\mathrm{E}_{1/2}(\partial K)=
\left\{\frac{x_h(u)+x_h(-u)}2:u\in\SSph^{n-1}\right\}.
\label{eq:Wigner-caustic-definition}
\end{equation}
The~two boundary points in each pair have opposite outer normals and
hence parallel tangent hyperplanes.
\end{definition}

\begin{proposition}[Central symmetral and Wigner hedgehog]
\label{prop:geometry-decomposition}
Let $K\subset\RR^n$ be a~convex body of class $C^2_+$ with support
function $h=\hodd+\heven$.  Then $\heven$ is the~support function of
the~central symmetral
\begin{equation}
C=\frac12\bigl(K+(-K)\bigr).
\label{eq:central-symmetral}
\end{equation}
Moreover,
\begin{align}
x_{\heven}(u)
&=
\frac{x_h(u)-x_h(-u)}2,
\label{eq:even-hedgehog}\\
x_{\hodd}(u)
&=
\frac{x_h(u)+x_h(-u)}2.
\label{eq:odd-hedgehog}
\end{align}
In particular, $x_{\hodd}(-u)=x_{\hodd}(u)$, and the~image of
$x_{\hodd}$ is $\mathrm{E}_{1/2}(\partial K)$, parametrized through the~double cover $\SSph^{n-1}\to\mathbb{RP}^{n-1}$.  We call the~resulting
parametrized object the~projective Wigner hedgehog; its geometric image
is the~Wigner caustic.
\end{proposition}

\begin{proof}
The~support function of $-K$ is $u\mapsto h(-u)$, so Minkowski
linearity of support functions proves \eqref{eq:central-symmetral}.
The~antipodal map has
differential $-\Id$.  Differentiating the~even and odd identities for
$\heven$ and $\hodd$, and identifying $T_u\SSph^{n-1}$ and
$T_{-u}\SSph^{n-1}$ with $u^\perp$, and using
\eqref{eq:hedgehog-param}, gives \eqref{eq:even-hedgehog} and
\eqref{eq:odd-hedgehog}.  The~last assertion follows immediately from
the~oddness of $\hodd$.
\end{proof}

The~Wigner caustic can be singular and self-intersecting.  Formula
\eqref{eq:mixed-volume-integral}, rather than an~unsigned Hausdorff
measure of its image, is
therefore the~appropriate source of its algebraic geometric
quantities.

Translations leave $\heven$ fixed and add a~linear function to
$\hodd$.  Thus
all mixed-volume expressions below are translation invariant.  Also,
$K$ has constant width $2r$ if and only if
\begin{equation}
\heven\equiv r,
\label{eq:constant-width-even-part}
\end{equation}
or equivalently if and only if $C=rB^n$.

\subsection{Parity}

\begin{lemma}[Parity of mixed volumes]
\label{lem:parity}
Suppose that $f_1,\ldots,f_n$ are formal differences of support
functions and that each is either even or odd on $\SSph^{n-1}$.  If
an~odd number of these functions are odd, then
$V(f_1,\ldots,f_n)=0$.
\end{lemma}

\begin{proof}
For $f\in C^2(\SSph^{n-1})$, let $\widehat f(u)=f(-u)$.  Simultaneous
reflection of all arguments leaves mixed volume unchanged, hence
\begin{equation*}
V(\widehat f_1,\ldots,\widehat f_n)
=V(f_1,\ldots,f_n).
\end{equation*}
If exactly $s$ functions are odd, multilinearity makes the~left-hand
side equal to $(-1)^sV(f_1,\ldots,f_n)$.  For odd $s$, the~mixed
volume equals its negative.
\end{proof}

\subsection{Regularity and approximation}

\noindent
The~differential expressions in
\eqref{eq:Q-def}--\eqref{eq:mixed-volume-integral} and \linebreak the~parametrized Wigner hedgehog
require the~stated smoothness.  The~convex-geometric defects,
projection identities, and inequalities do not.  Indeed, support
functions identify Hausdorff convergence of convex bodies with
uniform convergence on the~sphere, and mixed volumes are continuous
under this convergence; see \cite[Sections~1.8 and~5.1]{Schneider}.
Convolution with a~smooth rotational approximate identity gives smooth
support functions and commutes with the~antipodal map, so the~even and
odd parts converge separately.  Adding a~vanishing Euclidean-ball
summand makes the~approximants $C^2_+$.

If $K$ has constant width $2r$, rotational convolution preserves the~identity $h(u)+h(-u)=2r$.  A~subsequent Minkowski interpolation with
$rB^n$ preserves the~same width and yields positive curvature.  Thus
every constant-width body can be approximated by $C^\infty$ bodies of
class $C^2_+$ and the~same width.  Consequently, all statements below
that involve only $W_j$, $\mathcal A_k$, projections, or mixed volumes
extend by continuity to arbitrary convex bodies.  Pointwise
differential or algebraic-volume interpretations are asserted only in
the~regularity class specified in their statements.

\section{Symmetrization defects and their exact expansion}

\noindent
For $t\in[-1,1]$, consider the~Minkowski symmetrization path
\begin{equation}
K_t
=
\frac{1+t}{2}K+\frac{1-t}{2}(-K).
\label{eq:symmetrization-path}
\end{equation}
Its support function is
\begin{equation}
h_t=\heven+t\hodd.
\label{eq:path-support}
\end{equation}
Thus $K_1=K$, $K_{-1}=-K$, and $K_0=C$.

\begin{definition}
\label{def:symmetrization-defect}
For $k=1,\ldots,n$, the~$k$-dimensional symmetrization defect is
\begin{equation}
\mathcal A_k(K)
=
W_{n-k}(C)-W_{n-k}(K).
\label{eq:defect-definition}
\end{equation}
We also put $\mathcal A_0(K)=0$.
\end{definition}

Notice that $\mathcal A_1=0$, because $W_{n-1}$ is Minkowski linear
and reflection invariant.  The~volume defect is $\mathcal A_n$, while
the~surface-area defect is $n\mathcal A_{n-1}$.

\begin{theorem}[Exact Wigner expansion]
\label{thm:exact-expansion}
Let $K\subset\RR^n$ be a~convex body, let $C$ be its central
symmetral, and write $h=\hodd+\heven$ as in
\eqref{eq:even-odd-decomposition}.  For $k=1,\ldots,n$,
\begin{align}
W_{n-k}(K_t)
&=
W_{n-k}(C)
+
\sum_{r=1}^{\lfloor k/2\rfloor}
\binom{k}{2r}t^{2r}
V\left(
\heven^{[k-2r]},\hodd^{[2r]},1^{[n-k]}
\right).
\label{eq:defect-path-expansion}
\end{align}
Consequently,
\begin{equation}
\mathcal A_k(K)
=
-
\sum_{r=1}^{\lfloor k/2\rfloor}
\binom{k}{2r}
V\left(
\heven^{[k-2r]},\hodd^{[2r]},1^{[n-k]}
\right).
\label{eq:defect-mixed-expansion}
\end{equation}
\end{theorem}

\begin{proof}
By \eqref{eq:quermass-normalization}, \eqref{eq:path-support}, and
multilinearity,
\begin{equation*}
W_{n-k}(K_t)
=
\sum_{s=0}^k
\binom{k}{s}t^s
V\left(\heven^{[k-s]},\hodd^{[s]},1^{[n-k]}\right).
\end{equation*}
Lemma~\ref{lem:parity} removes every odd value of $s$.  The~term
$s=0$ is $W_{n-k}(C)$.  Setting $t=1$ and using
\eqref{eq:defect-definition} proves \eqref{eq:defect-mixed-expansion}.
\end{proof}

Formula \eqref{eq:defect-path-expansion} gives a~variational meaning
to every mixed term: it is
an~even response of a~quermassintegral when one moves from the~central
symmetral in the~Wigner direction $\hodd$.

\begin{theorem}[Positivity, rigidity, and the~quadratic response]
\label{thm:positive-defects}
Let $K\subset\RR^n$ be a~convex body and $C$ its central symmetral.
For $k=2,\ldots,n$,
\begin{equation}
\mathcal A_k(K)\geqslant0.
\label{eq:defect-positive}
\end{equation}
Equality holds if and only if $K$ is centrally symmetric up to
translation.  Moreover,
\begin{equation}
V\left(\heven^{[k-2]},\hodd^{[2]},1^{[n-k]}\right)\leqslant0.
\label{eq:quadratic-response-sign}
\end{equation}
\end{theorem}

\begin{proof}
The~Brunn--Minkowski inequality for quermassintegrals states that
\begin{equation}
W_{n-k}((1-s)L+sM)^{1/k}
\geqslant
(1-s)W_{n-k}(L)^{1/k}
+sW_{n-k}(M)^{1/k};
\label{eq:BM-quermass}
\end{equation}
see \cite[Theorem~7.4.5]{Schneider} for the~inequality and
\cite[Theorems~7.4.6 and~7.6.9]{Schneider} for the~equality case.
Apply \eqref{eq:BM-quermass} to $L=K$,
$M=-K$, and $s=1/2$.  Reflection invariance gives
$W_{n-k}(-K)=W_{n-k}(K)$, and therefore
\eqref{eq:defect-positive}.

For $k\geqslant2$, equality in \eqref{eq:BM-quermass} holds precisely when
$K$ and $-K$
are homothetic.  They have the~same size, so the~homothety ratio is
one and they differ by a~translation.  This is equivalent to central
symmetry up to translation.

Finally, $t\mapsto W_{n-k}(K_t)^{1/k}$ is concave by
\eqref{eq:BM-quermass} and even by
reflection invariance.  Hence its second derivative at $t=0$ is
non-positive.  The~first derivative of $W_{n-k}(K_t)$ vanishes at
zero, while \eqref{eq:defect-path-expansion} gives
\begin{equation*}
\left.\frac{\dd^2}{\dd t^2}\right|_{t=0}
W_{n-k}(K_t)
=
k(k-1)
V\left(\heven^{[k-2]},\hodd^{[2]},1^{[n-k]}\right).
\end{equation*}
This proves \eqref{eq:quadratic-response-sign}.
\end{proof}

\begin{remark}
The~total sum in \eqref{eq:defect-mixed-expansion} has a~prescribed
sign, but the~higher individual
terms generally do not.  This point is already essential in dimension
four, where a~quartic Wigner term occurs.
\end{remark}

\section{Projection formulas and Rogers--Shephard bounds}

\noindent
Let $\Gr(n,k)$ be the~Grassmannian of $k$-dimensional linear
subspaces of $\RR^n$, equipped with invariant probability measure
$\dd E$.  We write $K\restr E$ for the~orthogonal projection of $K$
onto $E$.

The~Kubota formula, with the~normalization
\eqref{eq:quermass-normalization}, reads
\cite[(5.72)]{Schneider}
\begin{equation}
W_{n-k}(K)
=
\frac{\kappa_n}{\kappa_k}
\int_{\Gr(n,k)}
\Vol_k(K\restr E)
\dd E.
\label{eq:Kubota}
\end{equation}

Orthogonal projection commutes with Minkowski addition and reflection,
so
\begin{equation*}
C\restr E
=
\frac{K\restr E+(-(K\restr E))}{2}.
\end{equation*}
Subtract \eqref{eq:Kubota} for $K$ from the~same formula for $C$ and we obtain the following proposition.

\begin{proposition}[Projection formula]
\label{prop:projection-defect}
Let $K\subset\RR^n$ be a~convex body and
$C=\frac12(K+(-K))$ its central symmetral.  For
$k=1,\ldots,n$,
\begin{equation}
\mathcal A_k(K)
=
\frac{\kappa_n}{\kappa_k}
\int_{\Gr(n,k)}
\left[
\Vol_k\left(\frac{K\restr E+(-(K\restr E))}{2}\right)
-\Vol_k(K\restr E)
\right]
\dd E.
\label{eq:projection-defect}
\end{equation}
\end{proposition}

Thus $\mathcal A_k$ is an~averaged $k$-dimensional asymmetry deficit.
The~pointwise integrand in \eqref{eq:projection-defect} is non-negative
by the~ordinary
Brunn--Minkowski inequality in $E$.

\begin{proposition}[Two-sided universal bound]
\label{prop:RS-bound}
Let $K\subset\RR^n$ be a~convex body.  For $k=1,\ldots,n$,
\begin{equation}
0\leqslant\mathcal A_k(K)
\leqslant
\left(
2^{-k}\binom{2k}{k}-1
\right)W_{n-k}(K).
\label{eq:RS-defect}
\end{equation}
For $n=k$, the~upper constant is sharp and equality is attained by
simplices.
\end{proposition}

\begin{proof}
Let $L=K\restr E\subset E$.  The~Rogers--Shephard difference-body
inequality \cite[Theorem~1]{RogersShephard} gives
\begin{equation*}
\Vol_k(L-L)
\leqslant
\binom{2k}{k}\Vol_k(L).
\end{equation*}
For full-dimensional $L\subset E$, equality holds if and only if $L$
is a~simplex.  This includes the~one-dimensional case, in which every
interval is a~simplex.
Since $(L-L)/2$ is the~central symmetral of $L$,
\begin{equation*}
0\leqslant
\Vol_k\left(\frac{L-L}{2}\right)-\Vol_k(L)
\leqslant
\left(2^{-k}\binom{2k}{k}-1\right)\Vol_k(L).
\end{equation*}
Integrating and using \eqref{eq:Kubota} proves
\eqref{eq:RS-defect}.  When $n=k$, this is exactly the~sharp
Rogers--Shephard inequality.
\end{proof}

\begin{proposition}[Difference-body coefficients in Wigner variables]
\label{prop:Krawtchouk-transform}
Let $K\subset\RR^n$ be a~convex body, write
$h_K=\heven+\hodd$, and set
\begin{equation*}
I_{2r}
=
V\left(\heven^{[n-2r]},\hodd^{[2r]}\right),
\qquad
r=0,\ldots,\left\lfloor\frac n2\right\rfloor.
\end{equation*}
For $j=0,\ldots,n$, define
\begin{equation*}
G_j(K)=V\left((-K)^{[j]},K^{[n-j]}\right)
\end{equation*}
and the~binary Krawtchouk polynomial
\begin{equation}
\operatorname{Kr}_m(j;n)
=
\sum_{\ell=0}^{m}
(-1)^\ell
\binom{j}{\ell}
\binom{n-j}{m-\ell},
\label{eq:Krawtchouk-definition}
\end{equation}
where binomial coefficients outside their natural range are understood
to be zero.  Then
\begin{equation}
G_j(K)
=
\sum_{r=0}^{\lfloor n/2\rfloor}
\operatorname{Kr}_{2r}(j;n)I_{2r}.
\label{eq:Godbersen-Krawtchouk}
\end{equation}
Conversely,
\begin{equation}
I_{2r}
=
\frac{1}{2^n\binom{n}{2r}}
\sum_{j=0}^{n}
\binom{n}{j}\operatorname{Kr}_{2r}(j;n)G_j(K).
\label{eq:inverse-Krawtchouk}
\end{equation}
Along the~symmetrization path \eqref{eq:symmetrization-path}, the~same
volume polynomial therefore has the~two representations
\begin{align}
\Vol_n(K_t)
=
\sum_{r=0}^{\lfloor n/2\rfloor}
\binom{n}{2r}t^{2r}I_{2r}
=
2^{-n}\sum_{j=0}^{n}
\binom{n}{j}(1+t)^{n-j}(1-t)^jG_j(K).
\label{eq:path-two-bases}
\end{align}
\end{proposition}

\begin{proof}
The~support functions of $K$ and $-K$ are, respectively,
$\heven+\hodd$ and $\heven-\hodd$.  Expanding $G_j(K)$ by
multilinearity shows that the~coefficient of
$V(\heven^{[n-m]},\hodd^{[m]})$ is exactly
$\operatorname{Kr}_m(j;n)$.  Lemma~\ref{lem:parity} removes the~odd
values of $m$ and proves \eqref{eq:Godbersen-Krawtchouk}.  The~standard
orthogonality relation
\begin{equation*}
\sum_{j=0}^{n}\binom{n}{j}
\operatorname{Kr}_m(j;n)\operatorname{Kr}_s(j;n)
=2^n\binom{n}{m}\,\delta_{ms}
\end{equation*}
follows by extracting the~coefficient of $z^mw^s$ from
\begin{equation*}
\sum_{j=0}^{n}\binom{n}{j}
\bigl((1-z)(1-w)\bigr)^j
\bigl((1+z)(1+w)\bigr)^{n-j}
=2^n(1+zw)^n.
\end{equation*}
It gives \eqref{eq:inverse-Krawtchouk}.  Finally, the~first equation of
\eqref{eq:path-two-bases} is the~case $k=n$ of
\eqref{eq:defect-path-expansion}, while the~second follows by expanding
$K_t=\frac{1+t}{2}K+\frac{1-t}{2}(-K)$ directly.
\end{proof}

Thus the~Wigner expansion and the~mixed difference-body expansion are
the~power-basis and Bernstein-basis descriptions of one polynomial.
The~termwise estimates
\begin{equation}
G_j(K)\leqslant\binom{n}{j}\Vol_n(K),
\qquad j=0,\ldots,n,
\label{eq:Godbersen-inequalities}
\end{equation}
are the~Godbersen inequalities.  Their higher-order difference-body
extension and its proof for anti-blocking bodies are developed in
\cite{Kotrbaty2025}; the~full family \eqref{eq:Godbersen-inequalities}
is proved in the~recent preprint
\cite[Theorem~1.1]{KotrbatyMouamine2026}.  Summing
\eqref{eq:Godbersen-inequalities} with weights $\binom{n}{j}$ gives
\begin{equation*}
\Vol_n(K-K)=2^nI_0
\leqslant
\sum_{j=0}^{n}\binom{n}{j}^2\Vol_n(K)
=\binom{2n}{n}\Vol_n(K),
\end{equation*}
so Proposition~\ref{prop:RS-bound} is their aggregated
Rogers--Shephard consequence when $k=n$.  Applying the~same argument
to every $k$-dimensional projection and integrating gives the~full
bound \eqref{eq:RS-defect}.

\subsection{Transfer of lower volume bounds through projections}

\noindent
The~projection formula also transfers any dimension-specific lower
volume bound for constant-width bodies to all higher dimensions.

\begin{theorem}[Projection-transfer principle]
\label{thm:projection-transfer}
Fix $2\leqslant k\leqslant n$ and suppose that every $k$-dimensional convex body
$L$ of constant width $2r$ satisfies
\begin{equation}
\Vol_k(L)\geqslant\mu_k\kappa_k r^k
\label{eq:lower-volume-ratio}
\end{equation}
for some constant $0<\mu_k\leqslant1$.  If
$K\subset\RR^n$ has constant width $2r$, then
\begin{equation}
\mathcal A_k(K)\leqslant
\kappa_n(1-\mu_k)r^k.
\label{eq:projection-transfer}
\end{equation}
\end{theorem}

\begin{proof}
For every $E\in\Gr(n,k)$, the~projection $K\restr E$ has constant
width $2r$ in $E$, while its central symmetral is $rB_E$.  Hence
\eqref{eq:lower-volume-ratio} gives
\begin{equation*}
\Vol_k(rB_E)-\Vol_k(K\restr E)
\leqslant(1-\mu_k)\kappa_k r^k.
\end{equation*}
Insert this pointwise estimate into
\eqref{eq:projection-defect}.  The~invariant measure on the~Grassmannian has total mass one.
\end{proof}

For $k=2$, the~Blaschke--Lebesgue theorem gives
$\mu_2=2(\uppi-\sqrt3)/\uppi$; see
\cite[Theorem~1]{BogoselMixed}.  Nishioka's recent three-dimensional
bound \cite[Theorem~1]{Nishioka2026}, written at width $2r$, gives
$\mu_3=8/11$.  Thus every constant-width body in $\RR^n$ satisfies
\begin{align}
\mathcal A_2(K)
&\leqslant\kappa_n\left(\frac{2\sqrt3}{\uppi}-1\right)r^2,
\label{eq:projection-transfer-planar}\\
\mathcal A_3(K)
&\leqslant\frac3{11}\kappa_n r^3.
\label{eq:projection-transfer-spatial}
\end{align}
The~second estimate will materially sharpen the~four-dimensional
outer bound in Section~\ref{sec:dimension-four}.

\subsection{The~projected Wigner-area formula}

\noindent
Assume temporarily that $K$ is of class $C^2_+$.  For
$E\in\Gr(n,2)$, let $h_{\mathrm{odd},E}$ be the~restriction of
$\hodd$ to the~unit circle of $E$.  The~Wigner caustic of
$\partial(K\restr E)$ is parametrized by $x_{h_{\mathrm{odd},E}}$.
Its full-cover algebraic area is
\begin{equation}
a(h_{\mathrm{odd},E})
=
\frac12\int_0^{2\uppi}
\left(h_{\mathrm{odd},E}^2-(h_{\mathrm{odd},E}')^2\right)\dd\theta.
\label{eq:full-cover-area}
\end{equation}
Because $x_{h_{\mathrm{odd},E}}$ is $\uppi$-periodic, the~oriented area
counted once is
\begin{equation}
A^*_{\Wig}(K\restr E)=\frac12a(h_{\mathrm{odd},E}).
\label{eq:projectively-reduced-area}
\end{equation}
The~function $h_{\mathrm{odd},E}$ is $\uppi$-antiperiodic.
Wirtinger's inequality therefore implies
$a(h_{\mathrm{odd},E})\leqslant0$.  The~first harmonic is the~harmless
translation mode.

\begin{theorem}[Quadratic Wigner projection formula]
\label{thm:Wigner-projection}
Let $K\subset\RR^n$ be a~convex body of class $C^2_+$ and let $C$ be
its central symmetral.  Then
\begin{equation}
\mathcal A_2(K)
=
\frac{2\kappa_n}{\uppi}
\int_{\Gr(n,2)}
\left|A^*_{\Wig}(K\restr E)\right|
\dd E.
\label{eq:Wigner-projection}
\end{equation}
In particular,
\begin{equation}
0\leqslant\mathcal A_2(K)\leqslant\frac12W_{n-2}(K).
\label{eq:A2-RS}
\end{equation}
\end{theorem}

\begin{proof}
Formula \eqref{eq:defect-mixed-expansion} with $k=2$ gives
\begin{equation*}
\mathcal A_2(K)=-V(\hodd,\hodd,1^{[n-2]}).
\end{equation*}
Kubota's formula in the~precise normalization
\eqref{eq:Kubota} -- that is, \cite[(5.72)]{Schneider} -- is polynomial
under Minkowski addition and hence extends by polarization from convex
support functions to virtual ones.  Therefore
\begin{equation}
V(\hodd,\hodd,1^{[n-2]})
=
\frac{\kappa_n}{\uppi}
\int_{\Gr(n,2)}a(h_{\mathrm{odd},E})\dd E.
\label{eq:A2-projected-mixed-volume}
\end{equation}
Now use \eqref{eq:projectively-reduced-area} and
$a(h_{\mathrm{odd},E})\leqslant0$.
Inequality \eqref{eq:A2-RS} is the~case $k=2$ of
Proposition~\ref{prop:RS-bound}.  Equivalently, it follows by
applying the~sharp planar estimate
$|A^*_{\Wig}(L)|\leqslant\Area(L)/4$ inside
\eqref{eq:Wigner-projection}.  This planar estimate is recorded and
derived from Rogers--Shephard in Corollary \ref{cor:planar-wigner-area-bound} below.
\end{proof}

The~extension of \eqref{eq:Wigner-projection} to non-smooth planar
projections can be phrased
using Schneider's middle hedgehog \cite{SchneiderMiddle}.  The~mixed
volume and projection formulas themselves remain valid without
smoothness.

\section{Strengthened isoperimetric inequalities from the~central symmetral}

\noindent
The~Alexandrov--Fenchel inequality
\cite[Theorem~7.3.1]{Schneider}, iterated along the~quermassintegrals,
gives the~generalized Urysohn inequalities: for every convex body
$L\subset\RR^n$ and $k=1,\ldots,n$,
\begin{equation}
W_{n-1}(L)^k
\geqslant
\kappa_n^{\,k-1}W_{n-k}(L).
\label{eq:generalized-Urysohn}
\end{equation}
For $k\geqslant2$, equality holds if and only if $L$ is a~Euclidean ball.
The~equality characterization follows from
\cite[Theorems~7.4.6 and~7.6.9]{Schneider}.

\begin{theorem}[Central-symmetral family of strengthened isoperimetric inequalities]
\label{thm:isoperimetric-family}
Let $K\subset\RR^n$ be a~convex body and let $C$ be its central
symmetral.  For every $k=2,\ldots,n$,
\begin{equation}
W_{n-1}(K)^k
-\kappa_n^{\,k-1}W_{n-k}(K)
\geqslant
\kappa_n^{\,k-1}\mathcal A_k(K).
\label{eq:isoperimetric-deficit}
\end{equation}
Equality holds if and only if $K$ has constant width.
\end{theorem}

\begin{proof}
Minkowski linearity and reflection invariance give
\begin{equation*}
W_{n-1}(C)=W_{n-1}(K).
\end{equation*}
Apply \eqref{eq:generalized-Urysohn} to $C$ and use
\begin{equation*}
W_{n-k}(C)=W_{n-k}(K)+\mathcal A_k(K).
\end{equation*}
Equality holds precisely when $C$ is a~ball.  By
\eqref{eq:constant-width-even-part}, this is
equivalent to constant width of $K$.
\end{proof}

Theorem~\ref{thm:isoperimetric-family} is a~genuine strengthening of
\eqref{eq:generalized-Urysohn}, because $\mathcal A_k\geqslant0$.  For
$k=2$, it takes the~especially
geometric form
\begin{equation}
W_{n-1}(K)^2
\geqslant
\kappa_n
\left(
W_{n-2}(K)
+\frac{2\kappa_n}{\uppi}
\int_{\Gr(n,2)}
\left|A^*_{\Wig}(K\restr E)\right|\dd E
\right).
\label{eq:isoperimetric-family-k2}
\end{equation}

\subsection{The~corrected Alexandrov--Fenchel chain}

\noindent
For convenience, set $\mathcal A_0=\mathcal A_1=0$.  Since
$W_j(C)=W_j(K)+\mathcal A_{n-j}(K)$,
the~entire Alexandrov--Fenchel chain for $C$ becomes the~following.

\begin{proposition}
\label{prop:corrected-AF}
Let $K\subset\RR^n$ be a~convex body and let $C$ be its central
symmetral.  For $j=0,\ldots,n-2$,
\begin{align}
\left(W_{j+1}(K)+\mathcal A_{n-j-1}(K)\right)^2
&\geqslant
\left(W_j(K)+\mathcal A_{n-j}(K)\right)\cdot
\left(W_{j+2}(K)+\mathcal A_{n-j-2}(K)\right).
\label{eq:corrected-AF}
\end{align}
Equality holds if and only if $K$ has constant width.
\end{proposition}

\begin{proof}
This is
$W_{j+1}(C)^2\geqslant W_j(C)W_{j+2}(C)$, followed by the~equality
characterization for the~Alexandrov--Fenchel inequality with the~unit ball among the~fixed arguments.  The~inequality is
\cite[Theorem~7.3.1]{Schneider}.  The~equality statement follows from
\cite[Theorems~7.4.6 and~7.6.9]{Schneider}.
\end{proof}

\subsection{The~classical isoperimetric inequality applied to
\texorpdfstring{$C$}{C}}

\noindent
The~volume and surface area of the~central symmetral are
\begin{align}
\Vol_n(C)
&=\Vol_n(K)+\mathcal A_n(K),
\label{eq:C-volume}\\
S(C)
&=S(K)+n\mathcal A_{n-1}(K).
\label{eq:C-surface}
\end{align}

\begin{theorem}[Symmetral surface--volume inequality]
\label{thm:symmetral-iso}
For every convex body $K\subset\RR^n$,
\begin{equation}
\left(S(K)+n\mathcal A_{n-1}(K)\right)^n
\geqslant
n^n\kappa_n
\left(\Vol_n(K)+\mathcal A_n(K)\right)^{n-1}.
\label{eq:symmetral-isoperimetric}
\end{equation}
Equality holds if and only if $K$ has constant width.
\end{theorem}

\begin{proof}
Substitute \eqref{eq:C-volume} and \eqref{eq:C-surface} into the~Euclidean isoperimetric inequality
\begin{equation*}
S(C)^n\geqslant n^n\kappa_n\Vol_n(C)^{n-1}.
\end{equation*}
For this normalization and the~equality case, see
\cite[Section~7.2]{Schneider}.
Equality holds precisely when $C$ is a~ball, equivalently when $K$ has
constant width.
\end{proof}

For $n=2$, $\mathcal A_{n-1}=\mathcal A_1=0$, so
\eqref{eq:symmetral-isoperimetric} is a~genuine
one-sided strengthening.  For $n\geqslant3$, both sides are corrected, and
\eqref{eq:symmetral-isoperimetric} should be understood as an~exact
Wigner rewriting of the~isoperimetric inequality of $C$.

\section{The~universal quadratic Wigner energy}

\noindent
The~defect $\mathcal A_2$ admits a~direct analytic formula in every
dimension.

\begin{theorem}[Spectral formula and rigidity]
\label{thm:spectral-A2}
Let $K\subset\RR^n$ be a~convex body of class $C^2_+$, let $C$ be its
central symmetral, and let $\hodd$ be the~odd part of its support
function.  Then
\begin{equation}
\mathcal A_2(K)
=
\frac{1}{n(n-1)}
\int_{\SSph^{n-1}}
\left(
|\nabla \hodd|^2-(n-1)\hodd^2
\right)
\dd\sigma.
\label{eq:A2-energy}
\end{equation}
It is non-negative and vanishes if and only if $\hodd$ is the~restriction of a~linear function, equivalently if and only if $K$ is
centrally symmetric up to translation.
\end{theorem}

\begin{proof}
For a~smooth function $g$, the~operator
$Q_g=\nabla^2_{\SSph^{n-1}}g+g\Id$ is the~spherical
radius-of-curvature operator defined in \eqref{eq:Q-def}.  The~polarized determinant satisfies
\begin{equation*}
D_{n-1}(Q_g,\Id^{[n-2]})
=
\frac{\tr Q_g}{n-1}
=
\frac{\Delta g+(n-1)g}{n-1}.
\end{equation*}
Using \eqref{eq:mixed-volume-integral} and integrating by parts gives,
for arbitrary smooth
$f,g$,
\begin{equation}
V(f,g,1^{[n-2]})
=
\frac{1}{n(n-1)}
\int_{\SSph^{n-1}}
\left((n-1)fg-\langle\nabla f,\nabla g\rangle\right)
\dd\sigma.
\label{eq:mixed-bilinear}
\end{equation}
Now use \eqref{eq:defect-mixed-expansion} with $k=2$.  Since $\hodd$
is odd, it has zero spherical mean.  The~sharp Poincar\'e inequality
on $\SSph^{n-1}$ is
\begin{equation}
\int_{\SSph^{n-1}}|\nabla \hodd|^2\dd\sigma
\geqslant
(n-1)\int_{\SSph^{n-1}}\hodd^2\dd\sigma.
\label{eq:Poincare-sphere}
\end{equation}
Its equality space is the~first eigenspace of
$-\Delta_{\SSph^{n-1}}$, namely the~restrictions of linear
functions.  Applying \eqref{eq:Poincare-sphere} to
\eqref{eq:A2-energy} proves non-negativity and rigidity.  A~linear
odd term is exactly a~translation of $K$.  The~sharp inequality,
sphere spectrum, and harmonic eigenspaces are recalled in
\cite[Chapter~2]{MullerSpherical}.
\end{proof}

Let
\begin{equation*}
\hodd=\sum_{\substack{\ell\geqslant1\\ \ell\text{ odd}}}
h_{\mathrm{odd},\ell}
\end{equation*}
be the~spherical-harmonic decomposition, where
\begin{equation*}
-\Delta h_{\mathrm{odd},\ell}
=
\lambda_\ell h_{\mathrm{odd},\ell},
\qquad
\lambda_\ell=\ell(\ell+n-2).
\end{equation*}
Then \eqref{eq:A2-energy} becomes
\begin{equation}
\mathcal A_2(K)
=
\frac{1}{n(n-1)}
\sum_{\substack{\ell\geqslant1\\ \ell\text{ odd}}}
\left(\lambda_\ell-(n-1)\right)
\|h_{\mathrm{odd},\ell}\|_{L^2}^2.
\label{eq:A2-spectral}
\end{equation}
The~$\ell=1$ summand is zero.  Translate $K$ so that
$h_{\mathrm{odd},1}=0$.
Since the~next admissible degree is $\ell=3$ and
\begin{equation*}
\lambda_3-(n-1)=2(n+2),
\end{equation*}
we obtain the~following sharp estimate.

\begin{corollary}[Quantitative central-asymmetry estimate]
\label{cor:stability}
Let $K\subset\RR^n$ be a~convex body of class $C^2_+$ and let
$\hodd$ be the~odd part of its support function.  After the~translation
mode has been removed,
\begin{equation}
\mathcal A_2(K)
\geqslant
\frac{2(n+2)}{n(n-1)}
\|\hodd\|_{L^2(\SSph^{n-1})}^2.
\label{eq:A2-stability}
\end{equation}
The~constant is sharp, and equality holds precisely when $\hodd$
belongs to the~degree-three spherical-harmonic eigenspace (including
$\hodd=0$).  Consequently,
\begin{align}
W_{n-1}(K)^2-\kappa_nW_{n-2}(K)
&\geqslant
\kappa_n\mathcal A_2(K)
\geqslant
\frac{2\kappa_n(n+2)}{n(n-1)}
\|\hodd\|_{L^2}^2.
\label{eq:combined-stability}
\end{align}
\end{corollary}

\begin{proof}
Estimate \eqref{eq:A2-stability} follows term by term from
\eqref{eq:A2-spectral}.  Equality in the~spectral estimate requires that only the~degree-three component be
present.  The~first inequality in \eqref{eq:combined-stability} is
the~$k=2$ case of
Theorem~\ref{thm:isoperimetric-family}.
\end{proof}

\section{Dimension two}
\label{sec:dimension-two}

\noindent
Let $K\subset\RR^2$ be a~smooth strictly convex body.  Write $L(K)$
and $A(K)$ for its perimeter and area.  Then
$W_1(K)=\frac12L(K)$,
$W_0(K)=A(K)$, and 
$\kappa_2=\uppi$.
The~full-cover algebraic area of the~Wigner hedgehog is
\begin{equation*}
V(\hodd,\hodd)
=
\frac12\int_0^{2\uppi}
\left(\hodd^2-(\hodd')^2\right)\dd\theta
=
2A^*_{\Wig}(K).
\end{equation*}
It is non-positive.  Consequently,
\begin{equation}
\mathcal A_2(K)
=
A(C)-A(K)
=
-V(\hodd,\hodd)
=
2\left|A^*_{\Wig}(K)\right|.
\label{eq:planar-defect}
\end{equation}
For an~arbitrary planar convex body, we use
\begin{equation*}
A^*_{\Wig}(K):=-\frac12\bigl(A(C)-A(K)\bigr)
\end{equation*}
as the~definition of its projectively reduced algebraic Wigner area.
Schneider's middle-hedgehog construction
\cite[Sections~2--3]{SchneiderMiddle} shows that this convention agrees
with the~oriented area above in the~smooth strictly convex case.  Thus
\eqref{eq:planar-defect} holds for every planar convex body with this
interpretation.

\begin{figure}[h]
\centering
\includegraphics[width=0.666\textwidth]{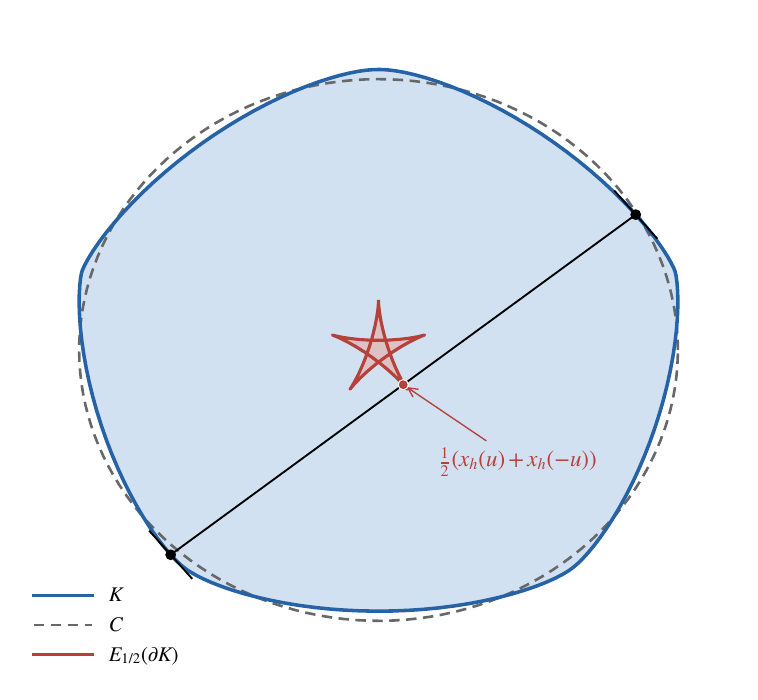}
\caption{The~convex body $K$ (blue), its central symmetral $C$
(dashed), and the~Wigner caustic $x_q(\SSph^1)$ (red).  The~black
chord joins two boundary points with opposite Gauss directions, and
the~short black segments indicate their parallel tangent lines.  The~
marked red midpoint belongs to the~Wigner caustic, which is traversed
twice by the~parametrization $x_q$}
\label{fig:n2-decomposition}
\end{figure}

Figure~\ref{fig:n2-decomposition} illustrates the~decomposition for
the~explicit support function
\begin{equation}
h(\theta)
=
2+0.10\cos(2\theta)+0.068\sin(5\theta),
\qquad
\theta\in\RR/2\uppi\ZZ.
\label{eq:n2-example-support}
\end{equation}
Its even and odd parts are, respectively,
$c(\theta)=2+0.10\cos(2\theta)$,
$q(\theta)=0.068\sin(5\theta)$.
Thus $c$ is the~support function of the~central symmetral $C$, whereas
$x_q(\SSph^1)$ parametrizes the~Wigner caustic.  Moreover,
\begin{equation}
h(\theta)+h''(\theta)
=2-0.30\cos(2\theta)-1.632\sin(5\theta)
\geqslant0.068,
\label{eq:n2-strict-convexity}
\end{equation}
so $h$ is the~support function of a~smooth strictly convex body.
The~marked points in the~figure are
$p_+=x_h(u_0)$,
$p_-=x_h(-u_0)$,
$u_0=(\cos\theta_0,\sin\theta_0)$ 
for
$\theta_0=0.73$.
Their outer normals are opposite, hence their tangent lines are
parallel, and linearity of \eqref{eq:hedgehog-param} gives
\begin{equation*}
\frac12(p_++p_-)=x_q(u_0)\in\mathrm{E}_{1/2}(\partial K).
\end{equation*}

Since $L(C)=L(K)$, Theorem~\ref{thm:isoperimetric-family} gives the following result for $n=k=2$.
\begin{theorem}\cite[Theorem~3.4]{Zwier2016}\label{thm:ImprovedIsoIneq}
Let $K$ be a convex planar body. Then
\begin{equation}
L(K)^2
\geqslant
4\uppi
\left(
A(K)+2\left|A^*_{\Wig}(K)\right|
\right).
\label{eq:planar-isoperimetric}
\end{equation}
Equality holds if and only if $K$ has constant width.
\end{theorem}

The improved isoperimetric inequality, stated here as Theorem~\ref{thm:ImprovedIsoIneq}, is the main result of \cite{Zwier2016}. Its original proof relies on the Fourier series expansion of the support function of a planar convex body.

Proposition~\ref{prop:RS-bound}, applied with $n=k=2$, gives a~sharp
global bound for the~algebraic Wigner area.

\begin{corollary}[Sharp planar Wigner-caustic area bound]
\label{cor:planar-wigner-area-bound}
Let $K\subset\RR^2$ be a~convex body.  Then the~algebraic area of its
projective Wigner hedgehog satisfies
\begin{equation}
-\frac14A(K)
\leq
A^*_{\Wig}(K)
\leq
0.
\label{eq:planar-wigner-area-bound}
\end{equation}
The~constant $\frac14$ is sharp.  Equality in the~left-hand inequality
holds if and only if $K$ is a~triangle, whereas equality in the~right-hand inequality holds if and only if $K$ is centrally symmetric
up to translation.
\end{corollary}

For smooth strictly convex bodies, the~left-hand inequality in
\eqref{eq:planar-wigner-area-bound} is strict, since such a body cannot
be a~triangle.  Nevertheless, the~constant $\frac14$ remains optimal,
as triangles can be approximated by smooth strictly convex bodies.

\section{Dimension three}
\label{sec:dimension-three}

\noindent
Let $K\subset\RR^3$ be of class $C^2_+$.  Denote by
\begin{equation*}
M(K)=\int_{\partial K}H\dd S
\end{equation*}
the~integral of the~normalized mean curvature.  With the~standard
normalization,
$S(K)=3W_1(K)$,
$M(K)=3W_2(K)$, and 
$\kappa_3={4\uppi}/{3}$.

\subsection{The~quadratic surface correction}

\noindent
On $\SSph^2$, let
$R_{\hodd}=\det Q_{\hodd}$.
Define the~non-negative full-cover algebraic Wigner surface defect by
\begin{equation*}
\mathscr S_{\Wig}(K)
=
-\int_{\SSph^2}R_{\hodd}\dd\sigma.
\end{equation*}
Because
\begin{equation*}
3V(\hodd,\hodd,1)=\int_{\SSph^2}R_{\hodd}\dd\sigma,
\end{equation*}
we have
\begin{equation}
\mathscr S_{\Wig}(K)=3\mathcal A_2(K).
\label{eq:Wigner-surface-defect}
\end{equation}
The~projection formula \eqref{eq:Wigner-projection} becomes
\begin{equation}
\mathscr S_{\Wig}(K)
=
8
\int_{\Gr(3,2)}
\left|A^*_{\Wig}(K\restr E)\right|
\dd E.
\label{eq:Wigner-surface-projection}
\end{equation}
Moreover,
\begin{equation}
S(C)=S(K)+\mathscr S_{\Wig}(K).
\label{eq:surface-central-symmetral-3d}
\end{equation}
The~quotient parameter space $\mathbb{RP}^2$ is non-orientable, so
\eqref{eq:Wigner-surface-projection}, rather than an~unsigned area of
the~image, is the~invariant
that enters naturally.

Figure~\ref{fig:n3-hedgehog} is based on the~axisymmetric support
function
\begin{equation}
h(u)=2+q(u),
\qquad
q(u)=\varepsilon f(u_3),
\qquad
\varepsilon=0.078,
\label{eq:n3-example-support}
\end{equation}
where
$f(t)=16t^5-20t^3+5t$.
Since $f$ is odd, we have
$h(u)+h(-u)=4$.
Thus $K$ has constant width $4$, its even support function is the~constant $c=2$, and its central symmetral is the~ball $C=2B^3$.

Both $K$ and its Wigner hedgehog $x_q(\SSph^2)$ are invariant under
rotations about the~$e_3$-axis.  Indeed, the~two principal eigenvalues
of $Q_f=\nabla^2_{\SSph^2}f+f\Id$ are
$-24f(t)$ and 
$f(t)-tf'(t)=-64t^5+40t^3$,
and both belong to $[-24,24]$.  Consequently,
$Q_h=
2\Id+\varepsilon Q_f
\geqslant
(2-24\varepsilon)\Id=
0.128\Id$,
so the~pictured body is smooth and strictly convex.

Let
$E=\operatorname{span}\{e_1,e_3\}$
be a~meridional plane.  Along the~great circle
$u(\theta)=(\cos\theta,0,\sin\theta)$
the~odd part reduces to
$q(u(\theta))
=
0.078\sin(5\theta)$.
Therefore the~Wigner caustic of the~orthogonal projection
$K_E=K\restr E$ has the~five-cusped pentagram shape shown in
Figure~\ref{fig:n3-hedgehog}(B).  The~middle part of panel~(A) is a~
magnified three-dimensional view of the~full axisymmetric hedgehog
$x_q(\SSph^2)$; the~dark red curve is its meridional generating curve.

This surface of revolution is deliberately highly nongeneric: axial
symmetry forces degeneracies that are absent after a~generic
perturbation.  Generic singularities of Wigner caustics and affine
equidistants in the~related Lagrangian and higher-dimensional surface
settings are classified or analyzed in
\cite{CraizerDomitrzRios,DomitrzJaneczkoRiosRuas,DomitrzManoelRios,
DomitrzRios,DomitrzRiosRuas,GiblinReeve}.

\begin{figure}[h]
\centering
\includegraphics[width=\textwidth]{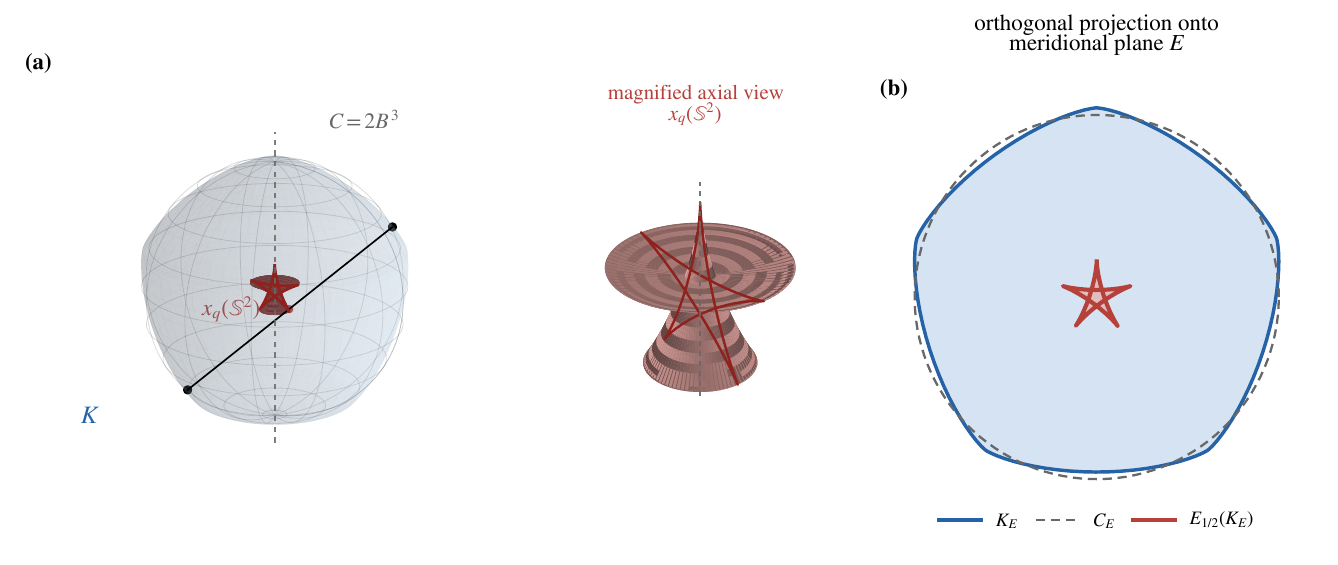}
\caption{An~axisymmetric constant-width body and its Wigner hedgehog
in dimension three.  \textup{(A)} Left: the~body $K$ (blue), its
spherical central symmetral $C=2B^3$ (gray wireframe), the~singular
Wigner hedgehog $x_q(\SSph^2)$ (red), and one antipodal Gauss chord.
Centre: a~magnified three-dimensional view of the~axisymmetric
hedgehog and its meridional generating curve.  \textup{(B)} The~orthogonal projection onto the~meridional plane
$E=\operatorname{span}\{e_1,e_3\}$, showing $K_E$, the~central disk
$C_E$, and the~pentagram-shaped planar Wigner caustic
$\mathrm{E}_{1/2}(\partial K_E)$}
\label{fig:n3-hedgehog}
\end{figure}

\subsection{The~volume correction}

\noindent
Parity gives
$V(\hodd,\hodd,\hodd)=0$
and
\begin{equation*}
\Vol_3(K)=\Vol_3(C)+3V(\heven,\hodd,\hodd).
\end{equation*}
Therefore
\begin{equation}
\mathcal A_3(K)
=
-3V(\heven,\hodd,\hodd)
=
-\int_{\SSph^2}\heven R_{\hodd}\dd\sigma.
\label{eq:volume-defect-3d}
\end{equation}
Thus $\mathcal A_3$ is not the~algebraic volume of the~Wigner
hedgehog.  That volume vanishes by parity.  It is instead the~Wigner
area density weighted by the~even support function $\heven$.

Theorem~\ref{thm:isoperimetric-family} now yields two genuine
refinements:
\begin{align}
M(K)^2
&\geqslant
4\uppi\left(S(K)+\mathscr S_{\Wig}(K)\right),
\label{eq:3d-isoperimetric-k2}\\
M(K)^3
&\geqslant
48\uppi^2
\left(\Vol_3(K)+\mathcal A_3(K)\right).
\label{eq:3d-isoperimetric-k3}
\end{align}
The~corrected surface--volume inequality is
\begin{equation}
\left(S(K)+\mathscr S_{\Wig}(K)\right)^3
\geqslant
36\uppi
\left(\Vol_3(K)+\mathcal A_3(K)\right)^2.
\label{eq:3d-symmetral-isoperimetric}
\end{equation}
All three equalities characterize constant width.

Proposition~\ref{prop:RS-bound} gives
\begin{equation}
0\leqslant\mathscr S_{\Wig}(K)\leqslant\frac12S(K),
\qquad
0\leqslant\mathcal A_3(K)\leqslant\frac32\Vol_3(K).
\label{eq:3d-RS-defects}
\end{equation}

\subsection{Constant width and the~Blaschke formula}

\noindent
If $K$ has constant width $2r$, then $\heven=r$ and
\eqref{eq:Wigner-surface-defect}, \eqref{eq:volume-defect-3d} give
$\mathcal A_3(K)=r\mathscr S_{\Wig}(K)$.
Since $C=rB^3$,
\begin{align}
S(K)
&=4\uppi r^2-\mathscr S_{\Wig}(K),
\label{eq:constant-width-surface-3d}\\
\Vol_3(K)
&=\frac{4\uppi}{3}r^3-r\mathscr S_{\Wig}(K).
\label{eq:constant-width-volume-3d}
\end{align}
Eliminating $\mathscr S_{\Wig}$ gives the~classical Blaschke relation
\begin{equation}
\Vol_3(K)
=
rS(K)-\frac{8\uppi}{3}r^3.
\label{eq:Blaschke-3d}
\end{equation}
The~preceding constant-width formulas are exactly formulas of Theorem~2 in \cite{AnciauxGuilfoyle} of Anciaux and Guilfoyle. They prove them using the different support-function calculations.  In the~notation used in \cite{AnciauxGuilfoyle,Nishioka2026}, the~quadratic
functional denoted there by $E$ is exactly
$\mathscr S_{\Wig}(K)$ when evaluated on $\hodd$.

\section{Dimension four}
\label{sec:dimension-four}

\noindent
Now let $K\subset\RR^4$ be of class $C^2_+$.  Since
$\kappa_4=\uppi^2/2$,
define the~mean width by
\begin{equation*}
\overline w(K)
=
\frac{2W_3(K)}{\kappa_4}
=
\frac{4}{\uppi^2}W_3(K).
\end{equation*}
The~three nontrivial defects are
\begin{align}
\mathcal A_2(K)
&=-V(\hodd,\hodd,1,1),
\label{eq:A2-4d}\\
\mathcal A_3(K)
&=-3V(\heven,\hodd,\hodd,1),
\label{eq:A3-4d}\\
\mathcal A_4(K)
&=-6V(\heven,\heven,\hodd,\hodd)
-V(\hodd,\hodd,\hodd,\hodd).
\label{eq:A4-4d}
\end{align}
In particular,
\begin{align}
S(C)
&=S(K)+4\mathcal A_3(K),
\label{eq:C-surface-4d}\\
\Vol_4(C)
&=\Vol_4(K)+\mathcal A_4(K).
\label{eq:C-volume-4d}
\end{align}

\subsection{Projected Wigner areas and curvature}

\noindent
The~three defects already have complementary integral-geometric
meanings.  Formula \eqref{eq:Wigner-projection} specializes to
\begin{equation}
\mathcal A_2(K)
=
\uppi
\int_{\Gr(4,2)}
\left|A^*_{\Wig}(K\restr E)\right|
\dd E.
\label{eq:A2-projection-4d}
\end{equation}
whereas Proposition~\ref{prop:projection-defect} gives
\begin{equation}
\mathcal A_3(K)=\frac{3\uppi}{8}
\int_{\Gr(4,3)}
\left[\Vol_3(C\restr E)-\Vol_3(K\restr E)\right]\dd E,
\label{eq:A3-projection-4d}
\end{equation}
and $\mathcal A_4(K)=\Vol_4(C)-\Vol_4(K)$.  Thus $\mathcal A_2$ is a~
mean projectively reduced Wigner area, $\mathcal A_3$ is a~mean
three-dimensional symmetrization gain, and $\mathcal A_4$ is the~full
four-dimensional gain.
If $H$ denotes normalized mean curvature of $\partial K$, then
\begin{equation*}
W_2(K)=\frac14\int_{\partial K}H\dd S.
\end{equation*}
The~$k=2$ case of the~central-symmetral family in
Theorem~\ref{thm:isoperimetric-family}, equivalently of
\eqref{eq:isoperimetric-deficit}, therefore becomes
\begin{equation}
\overline w(K)^2
\geqslant
\frac{2}{\uppi^2}
\int_{\partial K}H\dd S
+\frac{8}{\uppi}
\int_{\Gr(4,2)}
\left|A^*_{\Wig}(K\restr E)\right|
\dd E.
\label{eq:mean-width-k2-4d}
\end{equation}

The~remaining two members are
\begin{align}
\overline w(K)^3
&\geqslant
\frac{4}{\uppi^2}
\left(S(K)+4\mathcal A_3(K)\right),
\label{eq:mean-width-k3-4d}\\
\overline w(K)^4
&\geqslant
\frac{32}{\uppi^2}
\left(\Vol_4(K)+\mathcal A_4(K)\right).
\label{eq:mean-width-k4-4d}
\end{align}
Equality in \eqref{eq:mean-width-k2-4d}--\eqref{eq:mean-width-k4-4d}
is equivalent to constant width.

The~symmetral surface--volume inequality becomes
\begin{equation}
\left(S(K)+4\mathcal A_3(K)\right)^4
\geqslant
128\uppi^2
\left(\Vol_4(K)+\mathcal A_4(K)\right)^3.
\label{eq:symmetral-isoperimetric-4d}
\end{equation}

\subsection{The~quartic Wigner volume}

\noindent
The~first three quadratic mixed volumes in \eqref{eq:A2-4d}--\eqref{eq:A4-4d} have fixed
signs:
\begin{align*}
V(\hodd,\hodd,1,1)&\leqslant0,\\
V(\heven,\hodd,\hodd,1)&\leqslant0,\\
V(\heven,\heven,\hodd,\hodd)&\leqslant0,
\end{align*}
by \eqref{eq:quadratic-response-sign} with $k=2,3,4$.  In contrast,
$V(\hodd,\hodd,\hodd,\hodd)$ has no separate
universal sign forced by the~argument.

The~Gauss-cover algebraic volume of the~Wigner hedgehog is
\begin{equation}
V_{\Wig}^{\mathrm{cov}}(K)
:=
V(\hodd,\hodd,\hodd,\hodd)
=
\frac14
\int_{\SSph^3}\hodd\det Q_{\hodd}\dd\sigma.
\label{eq:Wigner-volume-cover}
\end{equation}
Since $x_{\hodd}(-u)=x_{\hodd}(u)$, it is natural to introduce the~projectively
reduced normalization
\begin{equation}
V^*_{\Wig}(K)=\frac12V_{\Wig}^{\mathrm{cov}}(K).
\label{eq:Wigner-volume-projective}
\end{equation}
Here $\mathbb{RP}^3$ is orientable and the~antipodal covering
$\SSph^3\to\mathbb{RP}^3$ is orientation preserving.  Thus the~factor
$1/2$ is the~natural projective normalization of the~algebraic volume
after fixing the~induced orientation.  This is an~algebraic quantity
with multiplicity.  No regular embedded hypersurface is assumed.  Formula
\eqref{eq:A4-4d} becomes
\begin{equation}
\mathcal A_4(K)
=
-6V(\heven,\heven,\hodd,\hodd)-2V^*_{\Wig}(K).
\label{eq:A4-projective-Wigner}
\end{equation}
Thus the~four-dimensional volume correction splits into a~quadratic
term weighted twice by the~central symmetral and a~pure projective
Wigner volume.

Positivity of $\mathcal A_4$ also gives the~quartic control
\begin{equation}
V(\hodd,\hodd,\hodd,\hodd)
\leqslant
-6V(\heven,\heven,\hodd,\hodd).
\label{eq:quartic-basic-control}
\end{equation}
Equality in \eqref{eq:quartic-basic-control} holds precisely when $K$
is centrally symmetric up to translation.

\subsection{A~four-dimensional Blaschke-type identity}

\noindent
Suppose that $K$ has constant width $2r$.  Then $\heven=r$ and
\begin{align}
\mathcal A_3(K)
&=3r\mathcal A_2(K),
\label{eq:A3-A2-constant-width-4d}\\
\mathcal A_4(K)
&=6r^2\mathcal A_2(K)-V(\hodd,\hodd,\hodd,\hodd).
\label{eq:A4-A2-constant-width-4d}
\end{align}
Since $C=rB^4$, one obtains
\begin{align}
S(K)
&=4\kappa_4r^3-12r\mathcal A_2(K),
\label{eq:surface-constant-width-4d}\\
\Vol_4(K)
&=\kappa_4r^4-6r^2\mathcal A_2(K)
+V(\hodd,\hodd,\hodd,\hodd).
\label{eq:volume-constant-width-4d}
\end{align}
Eliminating $\mathcal A_2$ yields
\begin{equation}
\Vol_4(K)
=
\frac r2S(K)
-\frac{\uppi^2}{2}r^4
+V(\hodd,\hodd,\hodd,\hodd).
\label{eq:Blaschke-4d}
\end{equation}
Equivalently, the~final term is $2V^*_{\Wig}(K)$.  Formula
\eqref{eq:Blaschke-4d}
shows exactly why the~three-dimensional Blaschke relation does not
extend to a~linear volume--surface identity in dimension four: a~genuine quartic Wigner invariant remains.

\subsection{The~quartic invariant assumes both signs}

\noindent
The~absence of a~sign assertion after \eqref{eq:A4-4d} is essential,
not merely
a~limitation of the~argument.  The~following explicit computation
also shows that the~phenomenon persists arbitrarily close to a~ball.

\begin{lemma}[Explicit quartic Wigner volumes of opposite signs]
\label{lem:quartic-signs}
On $\SSph^3\subset\RR^4$, let
\begin{align}
q_-(u)
=u_1^3-\frac12u_1,
\qquad
q_+(u)
=(u_1^2-u_2^2)u_3-2u_1u_2u_4.
\label{eq:cubic-harmonics}
\end{align}
Both are odd spherical harmonics of degree three, and
\begin{equation}
\begin{array}{c|cc}
q&-V(q,q,1,1)=\|q\|_{L^2(\SSph^3)}^2&V(q,q,q,q)\\
\hline
q_-&\displaystyle\frac{\uppi^2}{32}&
\displaystyle-\frac{3\uppi^2}{128}\\[5pt]
q_+&\displaystyle\frac{\uppi^2}{12}&
\displaystyle\frac{\uppi^2}{7}
\end{array}
\label{eq:cubic-harmonic-values}
\end{equation}
Consequently, for every $r>0$ and all sufficiently small nonzero
$\varepsilon$, each function
\begin{equation*}
h_{\varepsilon,\pm}=r+\varepsilon q_\pm
\end{equation*}
is the~support function of a~$C^\infty$ body of class $C^2_+$ and
constant width $2r$.  Its quartic invariant is
$\varepsilon^4V(q_\pm,q_\pm,q_\pm,q_\pm)$.  Hence both signs occur
among constant-width bodies arbitrarily close to $rB^4$.
\end{lemma}

\begin{proof}
The~functions in \eqref{eq:cubic-harmonics} are the~restrictions of
the~harmonic cubics
\begin{equation*}
H_-(x)=x_1^3-\frac12x_1|x|^2,
\qquad
H_+(x)=\operatorname{Re}\bigl((x_1+\mathrm{i}x_2)^2
(x_3+\mathrm{i}x_4)\bigr).
\end{equation*}
For any homogeneous cubic $H$ and $q=H\restr\SSph^3$, the~spherical
Hessian formula gives
\begin{equation}
Q_q
=
\left.
P_u\bigl(D^2H(u)-2H(u)\Id\bigr)P_u
\right|_{u^\perp},
\qquad
V(q^{[4]})
=
\frac14\int_{\SSph^3}q\det Q_q\dd\sigma,
\label{eq:homogeneous-cubic-formula}
\end{equation}
where $P_u$ is orthogonal projection onto $u^\perp$.  This is also a~direct specialization of \eqref{eq:Q-def} and
\eqref{eq:mixed-volume-integral}.

For $q_-$, put $t=u_1$.  The~eigenvalues of $Q_{q_-}$ are
$6t-8t^3$, $-2t^3$, and $-2t^3$.
Consequently,
\begin{align*}
V(q_-^{[4]})
&=\int_{\SSph^3}
\left(10t^{10}-8t^{12}-3t^8\right)\dd\sigma
=10\frac{21\uppi^2}{256}
-8\frac{33\uppi^2}{512}
-3\frac{7\uppi^2}{64}
=-\frac{3\uppi^2}{128}.
\end{align*}
Here we used the~standard moment identity
\begin{equation*}
\int_{\SSph^3}t^{2m}\dd\sigma
=
2\uppi^2
\frac{1\cdot3\cdot\ldots\cdot (2m-1)}
{4\cdot6\cdot\ldots\cdot (2m+2)},
\qquad m\geqslant1.
\end{equation*}
The~same identity gives $\|q_-\|_2^2=\uppi^2/32$.

For completeness, the~positive example can be evaluated without a~computer algebra system.  Use Hopf coordinates
\begin{equation*}
u_1+\mathrm{i}u_2=a\,\mathrm{e}^{\mathrm{i}\alpha},
\qquad
u_3+\mathrm{i}u_4=b\,\mathrm{e}^{\mathrm{i}\beta},
\qquad
a=\cos\eta,\quad b=\sin\eta,
\end{equation*}
and put $\psi=2\alpha+\beta$, $C_\psi=\cos\psi$,
$S_\psi=\sin\psi$.  Then $q_+=a^2bC_\psi$.  In the~orthonormal frame
$(\partial_\eta,a^{-1}\partial_\alpha,b^{-1}\partial_\beta)$,
away from the~two harmless coordinate circles represented by $a=0$ and $b=0$, direct differentiation
gives
\begin{equation}
Q_{q_+}=
\begin{pmatrix}
2b(4b^2-3)C_\psi &(4b^2-2)S_\psi&2abS_\psi\\
(4b^2-2)S_\psi&2b(b^2-2)C_\psi&-2aC_\psi\\
2abS_\psi&-2aC_\psi&-2ba^2C_\psi
\end{pmatrix}.
\label{eq:Hopf-matrix}
\end{equation}
Writing
$\det Q_{q_+}=C_\psi^3F(b)+C_\psi S_\psi^2G(b)$,
one obtains
$F(b)=8b(3-4b^2)(1-b^2)^3$ and 
$G(b)=24b(1-b^2)^3$,
and hence
$3F(b)+G(b)=96b(1-b^2)^4$.
With angle brackets denoting angular mean,
$\langle C_\psi^4\rangle=3/8$,
$\langle C_\psi^2S_\psi^2\rangle=1/8$, and
$\dd\sigma=ab\,\dd\eta\dd\alpha\dd\beta
=b\,\dd b\dd\alpha\dd\beta$,
\eqref{eq:Wigner-volume-cover} gives
\begin{align*}
V(q_+^{[4]})
&=\uppi^2\int_0^1
12b^3(1-b^2)^5\dd b
=\frac{\uppi^2}{7}.
\end{align*}
The~same coordinates give, explicitly,
\begin{equation*}
\|q_+\|_2^2
=2\uppi^2\int_0^1b^3(1-b^2)^2\dd b
=\frac{\uppi^2}{12}.
\end{equation*}
Finally, the~first equality in the~middle column of
\eqref{eq:cubic-harmonic-values} follows from
\eqref{eq:A2-spectral}, because degree three has eigenvalue $15$ on
$\SSph^3$.

The~continuous fields $Q_{q_\pm}$ are bounded.  Put
\begin{equation*}
M_\pm=\max_{\SSph^3}\|Q_{q_\pm}\|_{\mathrm{op}}.
\end{equation*}
Then
$Q_{h_{\varepsilon,\pm}}=r\Id+\varepsilon Q_{q_\pm}$ is positive
definite whenever $|\varepsilon|<r/M_\pm$.  Oddness gives
$h_{\varepsilon,\pm}(u)+h_{\varepsilon,\pm}(-u)=2r$, completing the~convexity and constant-width assertions.
\end{proof}

\begin{proposition}[The~exact cubic-harmonic plane]
\label{prop:cubic-harmonic-plane}
Let $f=q_-$ and $g=q_+$ be given by
\eqref{eq:cubic-harmonics}, put $q_{\alpha,\beta}=\alpha f+\beta g$,
and write
$\mathscr E(q)=-V(q,q,1,1)$.
Then
\begin{align}
\mathscr E(q_{\alpha,\beta})
&=\uppi^2\left(\frac{\alpha^2}{32}+\frac{\beta^2}{12}\right),
\label{eq:cubic-plane-energy}\\
V(q_{\alpha,\beta}^{[4]})
&=\uppi^2\left(
-\frac{3\alpha^4}{128}
+\frac{3\alpha^2\beta^2}{70}
+\frac{\beta^4}{7}
\right).
\label{eq:cubic-plane-quartic}
\end{align}
For every nonzero $q_{\alpha,\beta}$,
\begin{equation}
-\frac{24}{\uppi^2}\mathscr E(q_{\alpha,\beta})^2
\leqslant V(q_{\alpha,\beta}^{[4]})
\leqslant\frac{144}{7\uppi^2}\mathscr E(q_{\alpha,\beta})^2.
\label{eq:cubic-plane-sharp-bounds}
\end{equation}
Both constants are optimal on this plane: equality on the~left occurs
exactly on the~$f$-axis and equality on the~right exactly on the~$g$-axis.  Furthermore,
\begin{equation}
V(q_{\alpha,\beta}^{[4]})=0,
\qquad
\frac{\beta^2}{\alpha^2}
=\frac{\sqrt{1194}-12}{80},
\label{eq:cubic-plane-zero-cone}
\end{equation}
describes the~two nonzero zero lines.  Consequently, for every $r>0$
there are nonspherical $C^\infty$ bodies of constant width $2r$,
arbitrarily close to $rB^4$, whose quartic Wigner volume vanishes.
\end{proposition}

\begin{proof}
The~orthogonal map
$(u_1,u_2,u_3,u_4)\mapsto(u_1,u_2,-u_3,-u_4)$ fixes $f$ and sends
$g$ to $-g$.  Orthogonal invariance and multilinearity of mixed
volume therefore eliminate $V(f^{[3]},g)$ and $V(f,g^{[3]})$.  The~same symmetry gives $\langle f,g\rangle_{L^2}=0$.  Since both
functions have degree three, \eqref{eq:A2-spectral} and
\eqref{eq:cubic-harmonic-values} prove
\eqref{eq:cubic-plane-energy}.

It remains to compute the~even cross term.  Let $F$ and $G$ be the~homogeneous harmonic cubics restricting to $f$ and $g$, respectively,
and set
\begin{equation*}
A_s(u)=D^2(F+sG)(u)-2(F+sG)(u)\Id.
\end{equation*}
For every symmetric $4\times4$ matrix $A$ and every unit vector $u$,
the~determinant of its bilinear restriction to $u^\perp$ is
$u^{\top}\operatorname{adj}(A)u$.  Thus
\eqref{eq:homogeneous-cubic-formula}, followed by direct polynomial
expansion, gives
\begin{equation}
\frac1{4\uppi^2}
\int_{\SSph^3}(F+sG)\,
u^{\top}\operatorname{adj}(A_s)u\dd\sigma
=-\frac3{128}+\frac3{70}s^2+\frac17s^4.
\label{eq:cubic-plane-adjugate-computation}
\end{equation}
For transparency, every coefficient in this expansion follows from
the~elementary moment formula
\begin{equation}
\frac1{\uppi^2}\int_{\SSph^3}
u_1^{2a_1}u_2^{2a_2}u_3^{2a_3}u_4^{2a_4}\dd\sigma
=
\frac{2}{(a_1+a_2+a_3+a_4+1)!}
\prod_{j=1}^4\frac{(2a_j)!}{4^{a_j}a_j!}.
\label{eq:S3-multimoment}
\end{equation}
Monomials with an~odd exponent integrate to zero.  The~identity follows
directly by writing a~Gaussian integral in Cartesian and polar
coordinates.  In particular, the~coefficient of $s^2$ in
\eqref{eq:cubic-plane-adjugate-computation} says
$6V(f,f,g,g)=3\uppi^2/70$.  Polarization and the~two pure values in
\eqref{eq:cubic-harmonic-values} now prove
\eqref{eq:cubic-plane-quartic}.

For the~normalized estimate, put $x=\beta^2/\alpha^2\geqslant0$ when
$\alpha\ne0$.  After cancelling the~common scale, the~ratio
$\uppi^2V(q_{\alpha,\beta}^{[4]})/\mathscr E(q_{\alpha,\beta})^2$
equals
\begin{equation*}
R(x)=
\frac{-3/128+3x/70+x^2/7}{(1/32+x/12)^2}.
\end{equation*}
Its derivative has the~sign of
${47}/{8960}+{3x}/{560}$,
so $R$ is strictly increasing on $[0,\infty)$.  Since
$R(0)=-24$ and $R(\infty)=144/7$, this proves
\eqref{eq:cubic-plane-sharp-bounds} and its equality statements.
Finally, the~numerator of \eqref{eq:cubic-plane-quartic} vanishes when
$640x^2+192x-105=0$.  Its unique positive root is the~value in
\eqref{eq:cubic-plane-zero-cone}.  The~strict-convexity argument at
the~end of Lemma~\ref{lem:quartic-signs} applies to any fixed
linear combination $q_{\alpha,\beta}$ and proves the~final assertion.
\end{proof}

\subsection{An~outer bound for the~quadratic--quartic feasible region}

\noindent
Although Lemma~\ref{lem:quartic-signs} rules out an~independent
sign for $V(\hodd^{[4]})$, convexity still couples it nontrivially to
the~quadratic energy.

\begin{theorem}[Necessary outer regions at constant width]
\label{thm:quartic-feasible}
Let $K\subset\RR^4$ be a~convex body of constant width $2r$, let $C$
be its central symmetral, and set
\begin{equation}
a=\mathcal A_2(K),
\qquad
b=V(\hodd^{[4]}),
\qquad
R=\kappa_4r^2.
\label{eq:abR-definition}
\end{equation}
Then
\begin{align}
W_2(K)=R-a,
\quad
W_1(K)=r(R-3a),
\quad\text{and}\quad
\Vol_4(K)=r^2(R-6a)+b.
\label{eq:quermass-abR}
\end{align}
The~classical Alexandrov--Fenchel and Rogers--Shephard inequalities
alone give the~necessary outer region
\begin{equation}
0\leqslant a<x_*R,
\qquad
6r^2a-\frac{27}{35}\kappa_4r^4
< b\leqslant
\frac{r^2a(R+3a)}{R-a}
\leqslant2r^2a,
\label{eq:classical-outer-region}
\end{equation}
where
\begin{equation}
x_*=\frac{101-4\sqrt{106}}{315}
=0.189896\ldots.
\label{eq:x-star}
\end{equation}
The~$j=1$ Godbersen inequality gives the~additional strict lower
bound
\begin{equation}
\frac{24}{5}r^2a-\frac35\kappa_4r^4<b.
\label{eq:Godbersen-lower-4d}
\end{equation}
Combining \eqref{eq:Godbersen-lower-4d} with the~three-dimensional
projection bound \eqref{eq:projection-transfer-spatial} strengthens
the~outer region to
\begin{equation}
0\leqslant a\leqslant\frac{R}{11},
\qquad
\frac{24}{5}r^2a-\frac35\kappa_4r^4
< b\leqslant
\frac{r^2a(R+3a)}{R-a}
\leqslant\frac75r^2a.
\label{eq:projected-outer-region}
\end{equation}
Equality in the~Alexandrov--Fenchel upper bound occurs only for a~ball.
Moreover, $a=0$ holds if and only if $K$ is
a~Euclidean ball up to translation, in which case $b=0$.
\end{theorem}

\begin{proof}
Equations \eqref{eq:quermass-abR} follow from
Lemma~\ref{lem:parity} and $V(\hodd,\hodd,1,1)=-a$, after expanding
$V\bigl((r+\hodd)^{[4-j]},1^{[j]}\bigr)$
for
$j=0,1,2,3$.

First, \eqref{eq:RS-defect} with $k=3$, which is the~projected
Rogers--Shephard inequality \cite[Theorem~1]{RogersShephard}, gives
\begin{equation*}
3ra=\mathcal A_3(K)\leqslant\frac32W_1(K)
=\frac32r(R-3a),
\end{equation*}
so $a\leqslant R/5$.  Next, the~adjacent Alexandrov--Fenchel inequality
\begin{equation}
W_1(K)^2\geqslant W_0(K)W_2(K),
\label{eq:AF-adjacent-4d}
\end{equation}
which is \cite[Theorem~7.3.1]{Schneider} with two unit-ball
arguments, together with \eqref{eq:quermass-abR}, yields
\begin{equation}
b\leqslant\frac{r^2a(R+3a)}{R-a}\leqslant2r^2a.
\label{eq:quartic-AF-upper}
\end{equation}
The~second inequality uses $a\leqslant R/5$.

Likewise, \eqref{eq:RS-defect} with $k=4$ gives
\begin{equation*}
6r^2a-b=\mathcal A_4(K)
\leqslant\frac{27}{8}\Vol_4(K).
\end{equation*}
Since $\Vol_4(K)=\kappa_4r^4-\mathcal A_4(K)$, this is equivalent to
the~lower bound for $b$ in \eqref{eq:classical-outer-region}.  Put
$x=a/R$.  Compatibility of that lower bound with the~first upper
bound in \eqref{eq:quartic-AF-upper} implies
\begin{equation}
315x^2-202x+27\geqslant0.
\label{eq:x-compatibility-polynomial}
\end{equation}
Together with $0\leqslant x\leqslant1/5$, this forces $x\leqslant x_*$, where
$x_*$ is the~smaller root in \eqref{eq:x-star}.  Equality
$x=x_*$ would make the~lower and upper bounds for $b$ coincide and
hence force equality in \eqref{eq:AF-adjacent-4d}.  The~equality case
of the~Alexandrov--Fenchel inequality
\cite[Theorems~7.4.6 and~7.6.9]{Schneider} would then make $K$ a~ball,
which has $a=0$.  Therefore $a<x_*R$, proving
\eqref{eq:classical-outer-region}.

The~equality statement in the~Rogers--Shephard difference-body
inequality says that equality is possible only for a~simplex
\cite[Theorem~1]{RogersShephard}.  A~body of constant width $2r$
satisfies $K-K=2rB^4$ and is strictly convex: otherwise a~non-singleton
support face $F(K,u)$, after subtracting any point of $F(K,-u)$, would
give a~non-singleton support face of the~ball $K-K$.  A~simplex is not
strictly convex.  Hence the~lower bound for $b$ is strict.
Likewise, the~equality case cited for
\eqref{eq:AF-adjacent-4d} proves that the~upper bound is strict unless
$K$ is a~ball.

The~case $j=1$ of \eqref{eq:Godbersen-inequalities} reads in dimension
four
\begin{equation*}
V(-K,K,K,K)\leqslant4\Vol_4(K).
\end{equation*}
Indeed, after translating the~centroid of $K$ to the~origin, the~classical inclusion $-K\subset4K$ and monotonicity of mixed volume
give this estimate.  Schneider's stability theorem
\cite{SchneiderStability} supplies the~equality characterization: a~full-dimensional equality case is a~simplex.
Since $K$ has constant width and $h_K=r+\hodd$, parity gives
\begin{equation*}
V(-K,K,K,K)=\kappa_4r^4-b.
\end{equation*}
Substitution of the~last identity and
$\Vol_4(K)=\kappa_4r^4-6r^2a+b$ yields
\eqref{eq:Godbersen-lower-4d}.  Equality in the~Godbersen inequality
would force $K$ to be a~simplex, so the~lower bound is strict for a~constant-width body.

For the~stronger projection-derived bound, combine
\eqref{eq:A3-A2-constant-width-4d} with
\eqref{eq:projection-transfer-spatial}:
\begin{equation*}
3ra=\mathcal A_3(K)\leqslant\frac3{11}\kappa_4r^3,
\end{equation*}
and hence $a\leqslant R/11$.  Substitution into the~last quotient in
\eqref{eq:quartic-AF-upper} gives
$(R+3a)/(R-a)\leqslant7/5$, proving
the~upper bounds in \eqref{eq:projected-outer-region}.  On the~same
range, the~difference between the~Godbersen lower bound and the~Rogers--Shephard lower bound is
\begin{equation*}
\frac6{35}r^2(R-7a)>0.
\end{equation*}
Thus \eqref{eq:Godbersen-lower-4d} supplies the~lower bound in
\eqref{eq:projected-outer-region}.  Finally, Theorem~
\ref{thm:positive-defects} shows that $a=0$ is equivalent to central
symmetry up to translation.  Central symmetry together with constant
width forces a~ball.  Its odd support component is linear, so
\eqref{eq:linear-kernel} gives $b=0$.
\end{proof}

\section{Constant-width expansions in arbitrary dimension}

\noindent
The~preceding low-dimensional identities are instances of one general
formula.  Suppose that $K\subset\RR^n$ has constant width $2r$, so
that
$h=r+\hodd$, where $\hodd(-u)=-\hodd(u)$.
For $p=0,\ldots,\lfloor n/2\rfloor$, define the~even Wigner
quermassintegrals
\begin{equation}
\mathcal Q_{2p}(\hodd)
=
V\left(\hodd^{[2p]},1^{[n-2p]}\right),
\qquad
\mathcal Q_0(\hodd)=\kappa_n.
\label{eq:even-Wigner-quermassintegrals}
\end{equation}
The~quadratic term satisfies
$\mathcal Q_2(\hodd)=-\mathcal A_2(K)\leqslant0$,
whereas no independent sign is available for all higher terms:
Lemma~\ref{lem:quartic-signs} gives both signs already for
$\mathcal Q_4$ in dimension four.

\begin{theorem}[All quermassintegrals at constant width]
\label{thm:CW-expansion}
Let $K\subset\RR^n$ be a~convex body of constant width $2r$, and write
$h_K=r+\hodd$ with $\hodd$ odd.  For $j=0,\ldots,n$,
\begin{equation}
W_j(K)
=
\sum_{p=0}^{\lfloor(n-j)/2\rfloor}
\binom{n-j}{2p}
r^{\,n-j-2p}
\mathcal Q_{2p}(\hodd).
\label{eq:constant-width-quermass-expansion}
\end{equation}
For $n\geqslant3$, volume and surface area satisfy
\begin{align}
\Vol_n(K)
=
\frac{r}{n-2}S(K)
-\frac{2\kappa_n}{n-2}r^n+
\frac{2}{n-2}
\sum_{p=2}^{\lfloor n/2\rfloor}
(p-1)\binom{n}{2p}
r^{\,n-2p}\mathcal Q_{2p}(\hodd).
\label{eq:constant-width-volume-surface}
\end{align}
\end{theorem}

\begin{proof}
Expand
\begin{equation*}
W_j(K)=V\left((r+\hodd)^{[n-j]},1^{[j]}\right)
\end{equation*}
by multilinearity.  Lemma~\ref{lem:parity} removes the~odd powers of
$\hodd$, giving \eqref{eq:constant-width-quermass-expansion}.

For $j=0$ and $j=1$, formula
\eqref{eq:constant-width-quermass-expansion} gives
\begin{align*}
\Vol_n(K)
=
\sum_{p=0}^{\lfloor n/2\rfloor}
\binom{n}{2p}r^{n-2p}\mathcal Q_{2p}(\hodd)
\quad\text{and}\quad
S(K)
=
n
\sum_{p=0}^{\lfloor(n-1)/2\rfloor}
\binom{n-1}{2p}r^{n-1-2p}\mathcal Q_{2p}(\hodd).
\end{align*}
Use
$n\binom{n-1}{2p}
=
(n-2p)\binom{n}{2p}$.
The~constant term and the~quadratic term are isolated by the~first two
terms on the~right-hand side of
\eqref{eq:constant-width-volume-surface}.  The~remaining coefficient is
\begin{equation*}
1-\frac{n-2p}{n-2}
=
\frac{2(p-1)}{n-2}.
\end{equation*}
This proves \eqref{eq:constant-width-volume-surface}.
\end{proof}

For $n=3$, the~sum in \eqref{eq:constant-width-volume-surface} is
empty and one recovers \eqref{eq:Blaschke-3d}.  For $n=4$, only
$p=2$ remains, and \eqref{eq:constant-width-volume-surface} is exactly
\eqref{eq:Blaschke-4d}.

\begin{corollary}[The~even-dimensional residual]
\label{cor:even-dimensional-residual}
Let $K\subset\RR^n$ be a~convex body of constant width $2r$, and write
$h_K=r+\hodd$ with $\hodd$ odd.  If $n=2m\geqslant4$, then
\begin{align}
\Vol_n(K)
=\frac{r}{n-2}S(K)-\frac{2\kappa_n}{n-2}r^n+
\frac{2}{n-2}
\sum_{p=2}^{m-1}(p-1)\binom{n}{2p}
r^{n-2p}\mathcal Q_{2p}(\hodd)+V(\hodd^{[n]}).
\label{eq:even-dimensional-residual}
\end{align}
The~final summand is the~Gauss-cover algebraic $n$-volume of the~Wigner hedgehog and occurs with coefficient one.  Its projectively
reduced normalization is $\frac12V(\hodd^{[n]})$, because
$\mathbb{RP}^{n-1}$ is orientable and the~antipodal cover is
orientation preserving.  If $n$ is odd, then
$V(\hodd^{[n]})=0$ by parity and $\mathbb{RP}^{n-1}$ is
non-orientable.
\end{corollary}

\begin{proof}
In \eqref{eq:constant-width-volume-surface}, the~coefficient of the~last term $p=m$ is
$\frac{2(m-1)}{2m-2}\binom{2m}{2m}=1$.
The~orientation assertions use the~standard criterion that
$\mathbb{RP}^d$ is orientable exactly when $d$ is odd.  Equivalently,
the~antipodal map on $\SSph^d$ has degree $(-1)^{d+1}$.  For odd $n$,
Lemma~\ref{lem:parity} annihilates $V(\hodd^{[n]})$.
\end{proof}

The~four-dimensional term in \eqref{eq:Blaschke-4d} is the~first
non-planar example of this residual, and
Lemma~\ref{lem:quartic-signs} shows that it is sign-indefinite
even on arbitrarily small constant-width perturbations of a~ball.

\section*{Declarations}

\subsection*{Funding}
\noindent
The~author received no financial support for the~research, authorship,
or publication of this article.

\subsection*{Conflict of interest}
\noindent
The~author declares that there are no competing interests.


\subsection*{Data availability}
\noindent
No data were generated or analyzed in this theoretical study.

\end{document}